\documentclass[11pt]{article}

\usepackage{calc}

\usepackage{graphicx}
\usepackage{amsmath,amssymb,epsfig}
\usepackage{amsmath,amsfonts,amssymb}
\usepackage{amsthm}
\usepackage{enumerate}
\usepackage{cancel}
\usepackage{multicol}
\usepackage{subfigure}

\usepackage{array}
\usepackage{bm}
\usepackage{boxedminipage}

\usepackage{path}
\usepackage{psfrag}
\usepackage{relsize}

   \newtheoremstyle{example}{\topsep}{\topsep}
     {}
     {}
     {\bfseries}
     {}
     {\newline}
     {\thmname{#1}\thmnumber{ #2}\thmnote{ #3}}

   \theoremstyle{example}

   \newtheoremstyle{conjecture}{\topsep}{\topsep}
     {}
     {}
     {\bfseries}
     {}
     {\newline}
     {\thmname{#1}\thmnumber{ #2}\thmnote{ #3}}

   \theoremstyle{conjecture}

   \newtheoremstyle{question}{\topsep}{\topsep}
     {}
     {}
     {\bfseries}
     {}
     {\newline}
     {\thmname{#1}\thmnumber{ #2}\thmnote{ #3}}

   \theoremstyle{question}

\newcommand{\nZ}{\mathbb Z}

\newcommand{\nCPP}{\mathbb {CP}^2}
\newcommand{\nCPB}{\overline{\mathbb {CP}}^2}
\newcommand{\toto}{T_o ^2 \times T_o ^2}
\newcommand{\seteq}{\stackrel{\text{\tiny def}}{=}}
\newcommand{\totopar}{T_o ^2 \times S^1 \cup S^1 \times T_o ^2}
\newcommand{\nRPP}{\mathbb {RP}^4}
\newcommand{\NT}{\nu_{TT}}
\newcommand{\SLIII}{SL(3; \nZ)}

\newcommand{\cS}{\mathcal S}

\newtheorem{thm}{\bf{Theorem}}[section]

\newtheorem{prop}{\bf{Proposition}}[section]

\newtheorem{rem}{Remark}

\newtheorem{conj}{Conjecture}[subsection]
\newtheorem{question}{Question}[subsection]

\newlength{\wid}
\newcommand{\myname}{Daniel Nash}

\title{New Homotopy 4-Spheres}

\author{\myname (dnash@renyi.hu)}

\begin{document}
\maketitle

\bibliographystyle{plain}

\begin{abstract} 
We use surgery along 2-tori embedded in a union of two copies of $\toto$ to produce a new collection of homotopy 4-spheres (4-manifolds homotopy equivalent to $S^4$ and hence homeomorphic to $S^4$ but possibly not diffeomorphic to $S^4$).  It is still unknown if these new examples are in fact exotic.  
\end{abstract}

\section{Introduction}
With the Poincar\'{e} Conjecture now established, the attention of many experts is shifting to the 4-dimensional smooth counterpart to the conjecture:
\begin{conj}[The Smooth Poincar\'{e} Conjecture in 4-D (SPC4)]\label{conj:SPC4} Let $M$ be a smooth 4-manifold homeomorphic to the 4-sphere $S^4$.  Then $M$ is diffeomorphic to  $S^4$.
\end{conj}

\noindent The persisting lack of any answer to SPC4 is probably in part due to the wild nature of smooth 4-manifolds in general, which --- even restricting our scope to the simply-connected setting --- have proven exceptionally formidable in terms of constructing any plenary classification scheme.
Still, of all simply-connected 4-manifolds, the 4-sphere continues to present perhaps the most elusive challenge when it comes to obtaining/finding exotic smooth structures.  
On the one hand, the literature abounds with \emph{potential} counterexamples to the conjecture; but on the other hand, not one example has yet been verified as exotic.
This is, it seems, largely due to the lack of any smooth invariant for $S^4$ (which other exotic 4-manifold constructions have relied upon).  

Historically, these exotic constructions of other simply-connected\\
4-manifolds have quite often made use of surgery along 2-tori (generalized logarithmic transformation).  This paper highlights the utility of torus surgery in conjunction with SPC4 and (once again) as a potential facet of the classification of smooth 4-manifolds in general.  In section \ref{sec:2} we lay out the background material needed to construct our examples.  Section \ref{sec:mainsec} comprises the heart of this work, the production of new homotopy 4-sphere examples (we do not however prove here that our examples are counterexamples to SPC4).  These constructions are inspired by an intriguing handlebody presentation of $S^4$ given by Fintushel-Stern \cite{FS10} and the role of surgery upon 2-tori embedded in $\toto$ as seen in Fintushel-Park-Stern's ``Reverse Engineering'' program \cite{FPS07}.

In section \ref{sec:bridge} we illustrate a further correspondence between SPC4 and surgery along embedded 2-tori in conjunction with the classic homotopy sphere examples of Cappell-Shaneson \cite{CS76b, CS76a} and the recent analysis of these examples by Gompf \cite{Go09}.  Specifically, their homotopy $S^4$'s can be viewed as the result of surgery along a circle in special mapping tori on $T^3$, later labelled $M_{\phi}$ and referred to as ``Cappell-Shaneson mapping tori''.  Gompf exhibited diffeomorphisms between all members of a certain family of homotopy spheres arising from $M_{\phi}$ manifolds by altering the monodromy via surgery along fishtail embedded 2-tori in the $T^3$ fiber.  

Our focus here is on the monodromy changing mechanics of torus surgeries.  
We exhibit a (perhaps) surprising connection between a sub-collection of our surgery manifolds produced from $\toto$ and the mapping tori $M_{\phi}$ of Cappell-Shaneson, but we also show that this approach does not directly imply the trivialization of our general homotopy sphere examples.

\section{Background Material}\label{sec:2}
\subsection{Definition: surgery along a torus}
Given a 4-manifold $M$ and a torus $T \subset M$ which has a trivial normal bundle $\nu T \subset M$,
a surgery (or generalized logarithmic transformation) along $T$ is the process of extracting the interior of a tubular neighborhood of $T$, and then regluing $T^2 \times D^2$ via some diffeomorphism $\delta$ of its boundary.  (The restriction on the normal bundle ensures $\nu T \approx T^2 \times D^2$.)  
Notice that the boundary of $T^2 \times D^2$ is $T^2 \times S^1 \approx T^3$, a three-torus; so diffeomorphisms of the boundary are elements of $GL(3, \nZ) \cong$ \emph{Diff}$(T^3)$.  The resulting manifold $M_{\delta, T}$ is given as:
$$M_{\delta, T} = (M \setminus \nu T) \cup_{\delta} T^2 \times D^2.$$
Due to the handlebody (see \cite{GS99, Ki89} for instance) structure of a trivial torus bundle $T^2 \times D^2 = h^0 \cup h_a ^1 \cup h_b ^1 \cup h^2$,

\ there is a unique way to attach the dualized 3- and 4- handles coming from $h_a ^1, h_b ^1, h^0$ to $(M \setminus \nu T) \cup h^2$.  Hence, the regluing map $\delta$ can be described by the attaching map of the 2-handle.  In terms of homology, this gluing of the 2-handle into the boundary of $(M \setminus \nu T)$ --- and the surgery itself --- depends on a choice of curves along the boundary.  Specifically, taking two loops $\left\{a', b' \right\}$ which generate $\pi_1(T)$ we push these in $\nu T$ out to loops $a$ and $b$ on the boundary of $M\setminus \nu T$.  If $\mu$ is the curve in $\nu T$ which bounds, then $B = \left\{\left[a\right], \left[b\right], \left[\mu\right] \right\}$ forms a basis for $H_1(\partial(M \setminus \nu T); \nZ) \cong H_1(T^3; \nZ) \cong \nZ^3$.  The surgery then can be defined by a linear combination in $B$ which gives the attaching curve for $\partial D^2$, the boundary of the attaching disk of $h^2 \approx D^2 \times D^2$.  In sum, the surgery map $\delta$ and the resulting manifold $M_{\delta, T}$ are given by (the choices $a$ and $b$ and):
\begin{align*}
\delta_* : H_1(T^2 \times \partial D^2; \nZ)& \longrightarrow H_1(\partial(M \setminus \nu T); \nZ), \ \text{which maps} \\
\delta_* : \left[ \partial D^2 \right]& \longmapsto p\left[\mu\right] +  q\left[a\right]  + r\left[b\right]
\end{align*}

Generally, one refers to the above as a $(p,q,r)$-surgery along $T$ with respect to $a,b$ or a degree $p$-surgery in the direction $qa + rb$.  (From now on we also denote both loops $a$ and their corresponding homology classes $\left[a\right]$ by simply ``$a$".)     For certain simpler situations (like those considered in this paper), one of the last two coefficients will be 0, and we will mimic the notation of Dehn surgery in 3-manifolds by calling this a ($\frac{p}{q}$)-surgery with $p$ the coefficient of the meridian.  

\subsection{``Reverse Engineering'' and torus surgeries of Fintushel-Stern}
Of particular import in this paper, is the approach of Fintushel and Stern \cite{FS10} (and earlier: Fintushel-Park-Stern \cite{FPS07}) in devising clever ways of discovering \emph{nullhomologous} tori embedded in standard 4-manifolds which are somehow linked to exotic smooth structures on these 4-manifolds. 

In \cite{FPS07} the authors define and implement their ``Reverse Engineering'' process whereby exotic smooth structures on small euler characteristic manifolds (e.g. $\nCPP \# n \nCPB, \ n\leq 8$) can be obtained.  In their description, a simply-connected manifold such as $M = \nCPP \# n \nCPB$ serves as the \emph{target} of the procedure, while a different non-simply-connected symplectic manifold--- the \emph{model} for $M$ --- is actually used as the starting point.  The above authors were able to produce an infinite collection of Seiberg-Witten invariant altering surgeries.

For the purposes of this paper, Taubes' result on Seiberg-Witten invariants \cite{Ta94} and its utilization as in \cite{FPS07} are not quite applicable.  On the other hand, this formulation of \emph{models} is indeed useful for our \emph{target}, $S^4$.

In fact, of the greatest use here is the Fintushel-Park-Stern model for a special target which is not a blow-up of $\nCPP$, the target $S^2 \times S^2$.  The model employed in \cite{FPS07} is a fiber sum along a genus two surface in two copies of $\Sigma_2 \times T^2$, that is $\Sigma_2 \times \Sigma_2$.  Also, if each genus-2 surface complement ($\Sigma_2 \times T^2 \setminus \Sigma_2 \times D^2$) is further decomposed as $\toto$ (a product of punctured 2-tori), then the (8-many) surgeries leading to a fake $S^2 \times S^2$ can be realized within the four individual $\toto$ copies.

\subsection{$\toto$}\label{sec:toto}
Now this decomposition of $\Sigma_2 \times \Sigma_2$ as a 4-fold union of $\toto$'s (equivalently, the complements of a the coordinate axes in copies of $T^4 = T^2 \times T^2$) suggests a key strategy for understanding exotic smooth structures and the related model manifolds might be to focus on this core building block $\toto$ itself.  Actually, Fintushel and Stern have arrived in this situation from an alternate starting point.  Pursuing useful nullhomologous tori embedded in standard 4-manifolds, Fintushel and Stern have exhibited a particular manifold-with-boundary (denoted by them as ``$A$") which itself contains a ``Bing double" of nullhomologous tori and upon which they build their models.  Figure \ref{fig:Amfld} below gives a handlebody diagram for $A$ which is equivalent to the one appearing in \cite{FS10}. 
			\begin{figure}[h]
			\begin{center}
			\includegraphics[width=8cm]{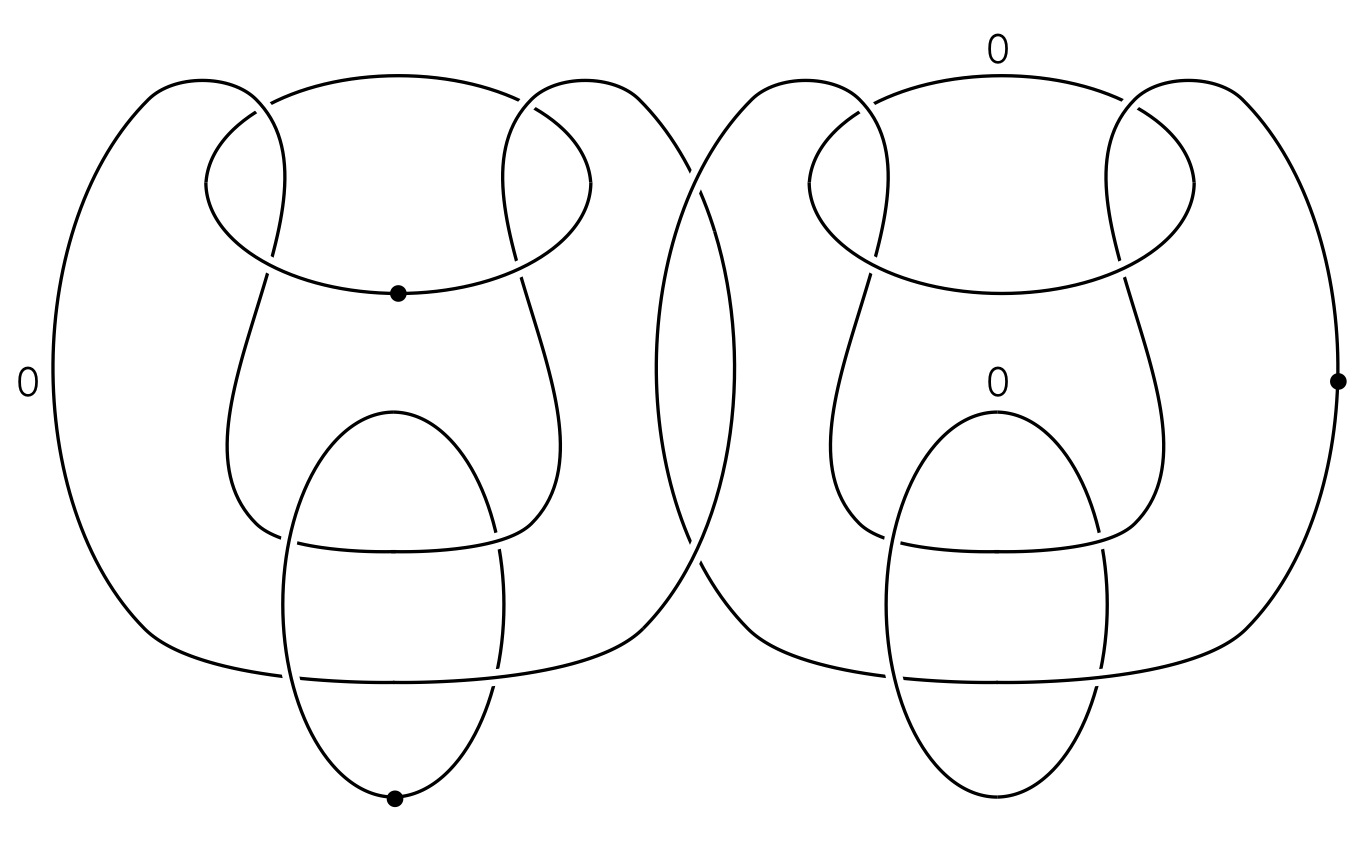}
			\caption{Fintushel and Stern's $A$ manifold}
			\label{fig:Amfld}
			\end{center}
			\end{figure}
Two key aspects concerning $A$ now become important for the emphasis of our work.  First, Fintushel-Stern have proven the following.  
(Let $B_T \seteq$ the pair of tori mentioned above, then:)
\begin{prop}[\cite{FS10}, see Proposition 2]\label{prop:FS}  The result of 0-framed surgery on the pair of tori $B_T \subset A$ is $\toto$. \hfill $\Box$
\end{prop}			
\noindent Second, Fintushel and Stern have also made the following observation which is simple to check:  If $\varphi$ is the involution of $\partial A$ which flips the handlebody's boundary about a vertical line through the middle of the diagram above, then $A \cup_{\varphi} \overline{A} \approx S^4$.  Forming this union amounts to gluing in the second copy's 2-handles as 0-framed meridians to the first copy's \emph{1-handles} and then attaching the dualized 1-handles as 3-handles.  Essentially, one arrives at the handlebody of figure \ref{fig:AuA} union three 3-handles and a 4-handle.
			\begin{figure}[h]
			\begin{center}
			\includegraphics[width=8cm]{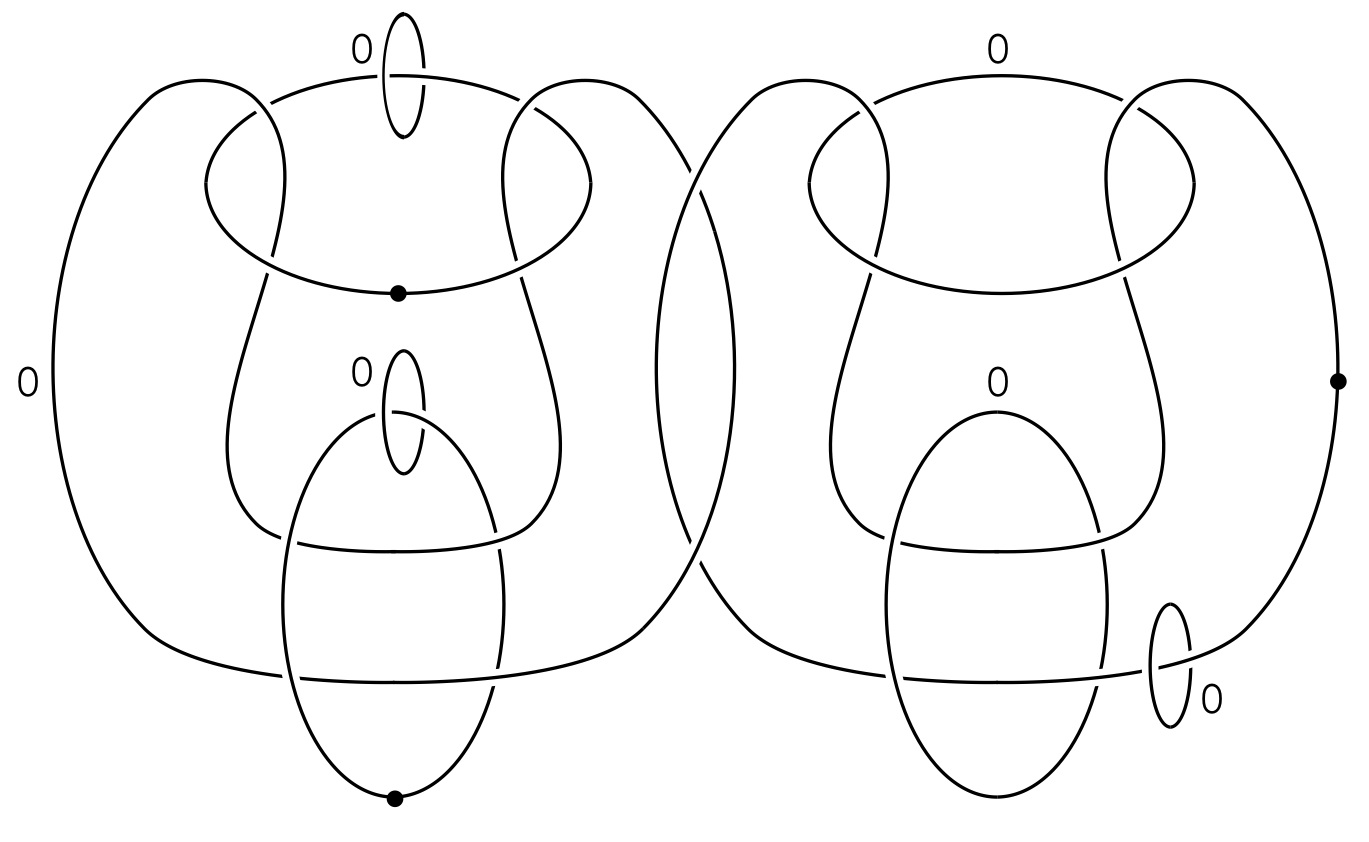}
			\caption{$A\cup_{\varphi} \overline{A} \ \setminus \ $ (3-handles, 4-handle)}
			\label{fig:AuA}
			\end{center}
			\end{figure}
			
After sliding 2-handles over these $0$-framed meridians and cancelling pairs of 1-,2-handles, the boundary is explicitly seen as $\#_3 S^1\times S^2$.  Thus, one can add back in the extra 3-handles, cancel the 2-,3-handle pairs, and add the 4-handle to obtain $S^4$.
These two observations above now give the connection between surgery on model manifolds and SPC4, and the way is paved for the main consideration of this section.  Overall, the above implies:
\begin{center} \textit{$S^4$ contains four nullhomologous tori, 0-framed surgery upon which yields $\toto \cup_{\varphi} \overline{\toto}$.}
\end{center}
Or dually, starting from the opposite direction: 
\begin{center}
\textit{$\toto \cup_{\varphi} \overline{\toto}$ contains four essential tori, surgery upon which  yields $S^4$}.
\end{center}

This leads one to consider whether \emph{other} surgeries upon tori in\\ $\toto \cup_{\varphi} \overline{\toto}$ will also produce $S^4$, or more importantly, whether there are surgeries that might possibly produce an \emph{exotic} $S^4$.  We exhibit below, surgeries which \textit{at least} produce a homotopy $S^4$ not a priori diffeomorphic to $A\cup_{\varphi} \overline{A}$.

\section{Homotopy 4-Spheres from $\toto$}\label{sec:mainsec}

\subsection{Constructing a New Homotopy 4-Sphere} 
To begin our construction, note that the boundary of $\toto$ 
is
	$$\partial(\toto) = \totopar, $$
where the two boundary terms are not disjoint but overlap in a torus.

In the following, we make use of the same convenient involution $\varphi$ which is a ``flip" along the entire boundary.  This can be formally defined by
\begin{align*}	
				\varphi : \totopar& \longrightarrow \totopar,\\
				&\varphi(x) = x^*, \ \
\end{align*}				
where for $x \in T_o ^2 \times S^1$, $x^*$ is the corresponding point of $S^1 \times T_o ^2$ and conversely.  Under this framework, we will prove the following result:

\begin{thm}[$\toto$ Surgery Theorem]\label{thm:totosurgery}
For $\varphi$ as above, there are two lagrangian tori in $\toto$ and a pair of lagrangian-framed surgeries such that the resulting surgery manifold $X'$ satisfies:
	$$X' \cup_{\varphi} \overline{X'} \cong S^4.$$
\end{thm}

\begin{proof}\ref{thm:totosurgery}
The surgeries in view here are actually performed identically in both copies.  For homotopy calculations we appeal to the results of Baldridge and Kirk \cite{BK08}, essentially surgering the same pair of tori depicted in their calculation.  In order to guarantee the effects of surgeries on $\pi_1$, we also utilize a slightly careful description of the torus surgeries.
Here label the $\pi_1$-generating loops passing through the basepoint $(x,y)$ in the $i^{th}$ punctured torus product by $a_i, b_i$ from one punctured torus factor and $c_i, d_i$ from the other.  Similar to \cite{FPS07} label lagrangian push-offs of these loops by ``primes" as in figure \ref{fig:totosquares}.
			\begin{figure}[h]
			\begin{center}
			\includegraphics[width=12cm]{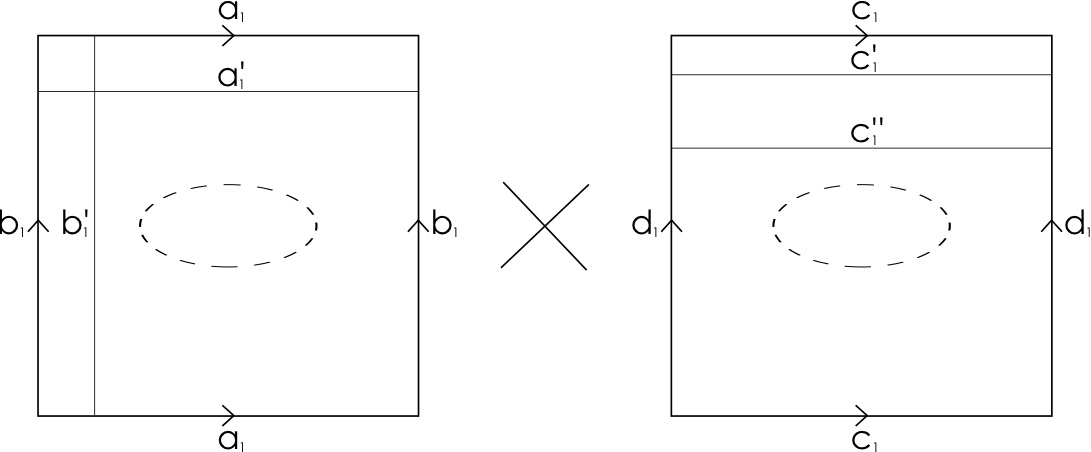}
			\caption{$\toto$ and lagrangian tori}
			\label{fig:totosquares}
			\end{center}
			\end{figure}

Set $T_{ac} = a_1' \times c_1'$ and $T_{bc} = b_1' \times c_1''$.  Then as in Baldridge-Kirk \cite{BK08} ``Theorem 2", the complement of the two tori in $\toto$ has fundamental group generated by $a_1, b_1, c_1, d_1$ with several relations.  These include:
	\begin{align}
		[a_1, c_1] &= 1\\
		[b_1, c_1] &= 1. \label{rel:b1c1}
	\end{align}
In the notation of ``Reverse Engineering" \cite{FPS07} the surgery tori, directions, and coefficients selected here are of the form  (torus, direction, coefficient).
After regluing the two tori along these surgery curves, we have new $\pi_1$ relations (again by ``Theorem 2'',  \cite{BK08}):
\begin{align*}
&\text{Surgery} \quad	 &\text{New\ \ } \pi_1 \text{\ \ relation}\\
&(a_1' \times c_1', \ a_1 ', -1) & [b_1 ^{-1}, d_1^{-1}] = a_1\\
&(b_1' \times c_1'', \ b_1', -1) &  [a_1 ^{-1}, d_1] = b_1\\
\end{align*}

Now to continue the proof of the theorem, we need:
\begin{prop}\label{prop:totoboundary}
Each of the loops $a_1,b_1,c_1,d_1$ are based homotopic to a corresponding loop on the boundary of $\toto$, in the complement of tori $\toto \setminus T_{ac} \cup T_{bc}$. 
\end{prop}
\begin{proof}\ref{prop:totoboundary}  For $a_1 ' \times {y}, \ b_1 ' \times {y}$, etc. push the corresponding point $x$ or $y$ along a straight linear path to $x'$ or $y'$.  This can be done in such a way that $\nu T_{ac}$ and $\nu T_{bc}$ are avoided as in figure \ref{fig:totopath}.
			\begin{figure}[h]
			\begin{center}
			\includegraphics[width=12cm]{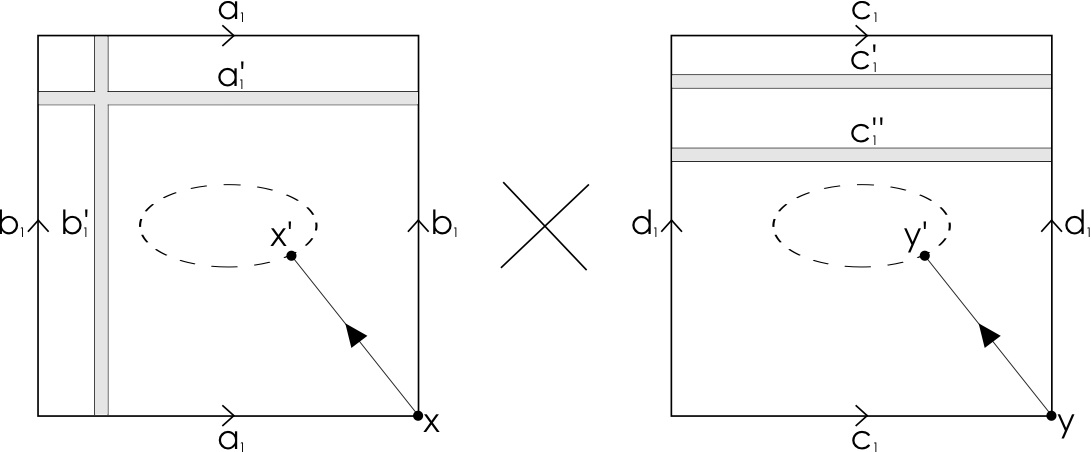}
			\caption{Paths from basepoint $(x,y)$ to the puncture}
			\label{fig:totopath}
			\end{center}
			\end{figure}

\end{proof}
With this proposition, in $\pi_1$ of the union of surgery manifolds $X' \cup_{\varphi} \overline{X'}$ we have the relations $a_1 \sim c_2$, $b_1 \sim d_2$,  $a_2 \sim c_1$, and $b_2 \sim d_1$.  Hence, applying the equivalent of  \eqref{rel:b1c1} to the second copy we also obtain
		\begin{align}
		[b_2, c_2] = [d_1, a_1] = 1\label{rel1}
		\end{align}
and from the first and second pair of surgeries
		\begin{align}
		 [b_1 ^{-1}, d_1^{-1}] &= a_1\label{rel2}  \\
		 [a_1 ^{-1}, d_1] &= b_1 \label{rel3} \\
		[b_2 ^{-1}, d_2^{-1}] = a_2 = c_1 &= [d_1 ^{-1}, b_1 ^{-1}] \label{rel4} \\
		 [a_2 ^{-1}, d_2] = b_2 = d_1 &= [c_1 ^{-1}, b_1] . \label{rel5}
		\end{align}

Using \eqref{rel1} together with \eqref{rel3} implies $b_1 = 1$, and then \eqref{rel2} and \eqref{rel4} in turn give $a_1, c_1 = 1$.  Finally, \eqref{rel5} gives $d_1 = 1$. 
After the four surgeries in the union $\toto \cup_{\varphi} \overline{\toto}$, we obtain the simply-connected manifold $\cS'  \seteq X'  \cup_{\varphi} \overline{X'}$.  Since $\cS'$ also has $\chi = 2$, it is therefore homeomorphic to $S^4$.  \hfill \emph{end proof theorem \ref{thm:totosurgery}} \end{proof}

\subsection{Families of Homotopy $S^4$'s}\label{sec:families}
By the choice of surgeries (in fact \emph{either} of the $-1$ or $+1$ surgeries works so that $X'$ as depicted above is only one such possible choice of surgery manifolds; it is not yet known whether these are pairwise diffeomorphic).
\ \ However any such $X'$ is distinct from $A$ (see section \ref{sec:bridge}), hence it is not obvious that $\cS'$ is standard $S^4$.  Now if one is willing to sacrifice the benefit of having a symplectic surgery manifold like $X'$, allowing a greater freedom in surgeries can still yield a homotopy $S^4$.

\begin{thm}[The Main Theorem]\label{prop:Surgeries} For $m,n \in \nZ$, Let $X_{m,n}$ denote the result of performing the ($\frac{m}{1}$)- and ($\frac{n}{1}$)-surgeries on $T_{ac}, T_{bc}$ and in the directions $a, b$ respectively.  Then the 4-manifold
$$\cS_{(m,n,m',n')} \seteq X_{m,n} \cup_{\varphi} \overline{X_{m',n'}}$$
is homeomorphic to $S^4$ for all $m,n,m',n' \in \nZ$.
\end{thm}

\begin{proof}\ref{prop:Surgeries} Replace the relations such as (\ref{rel3}) above with $[a_1 ^{-1}, d_1]^n = b_1$, etc.  Again this gives $b_1=1$ and in turn all three of the other generators are trivial as before. 
\end{proof}

\begin{rem}In the following section we will see that even the non-symplectic surgery manifolds $X_{m,n}$ above, when also both $m,n \neq 0$, are distinct from $A$, for slightly more subtle reasons (see proposition \ref{prop:surgRecast}).
\end{rem}

Furthermore, we have described these surgeries and surgery coefficients from the starting point and point of view of  $\toto \cup_{\varphi} \overline{\toto}$.  However, our specific pairs of tori: $a_i \times c_i, \ b_i \times c_i$ were precisely those producing $A$ (see \cite{FS10}) as well, and if $W_T$ is a tubular neighborhood of the union of the four surgery tori, then:
$$\toto \cup_{\varphi} \overline{\toto} \setminus W_T =  A\cup_{\varphi} \overline{A} \setminus W_T = S^4 \setminus W_T.$$
Hence, some regluing of four tori embedded in $S^4$ gives the manifold $\cS_{(m,n,m',n')}$, and this proves:

\begin{prop}  There is a four-parameter family of homotopy four-spheres, $\cS_{(m,n,m',n')}$, obtained by surgery along nullhomologous tori in $S^4$. \hfill $\Box$
\end{prop}

\section{Further Analysis and Final Remarks}\label{sec:bridge}
Now of course $\toto$ is nothing other than the 4-torus, viewed as $T^2 \times T^2$, with its coordinate axis 2-tori ($\NT \seteq \nu(S^1 _a \times S^1 _b  \cup S^1 _c \times S^1 _d)$) deleted.   Recombining the surgery manifolds $X_{m,n}$ with $\NT$ then gives the result of performing the same pair of surgeries in $T^4$.  Two consequences emerge from this.  First:
\begin{prop}\label{prop:surgRecast} The Fintushel-Stern manifold $A$ and the surgery manifolds $X_{m,n}$ satisfy:
\begin{enumerate}
	\item $A = X_{m,0} = X_{0,n}, \  \forall m,n \in \nZ$
	\item $X_{m,n} \ncong A$ if both $m, n \in \nZ \neq 0$.
\end{enumerate}
\end{prop}

\begin{proof}  We prove the above by recasting the pair of surgeries in $T^4 = \toto \cup \NT$ as surgeries in $T^4 = T^3 \times S^1$.  A similar trick was used already by Akhmedov-Baykur-Park in \cite{ABP08}.  In particular, note that regluings of both torus surgeries $(a' \times c', \ a ', m)$ and $(b' \times c'', \ b', n)$ are trivial on the $c$-factor, that is, the surgery maps are equivalent to 
$$(\text{Dehn-surgery on a loop in $T^3$}) \times Id |_{S^1}$$
in $T^4$ viewed as $T^3 \times S^1 = (a \times b \times d) \times c$.  We can then fully depict the surgery manifolds (union $\NT$) by taking the cartesian product of $S^1$ and the resulting 3-manifold, $Y_{m,n}$,  obtained from $T^3$ after the pair of Dehn-surgeries.  This is depicted in figure \ref{fig:XmnUnuTT}(a) where Dehn surgery is performed along push-offs of two of the meridians to the 0-framed Borromean link.
			\begin{figure}[h]
			\begin{center}
			\includegraphics[width=12cm]{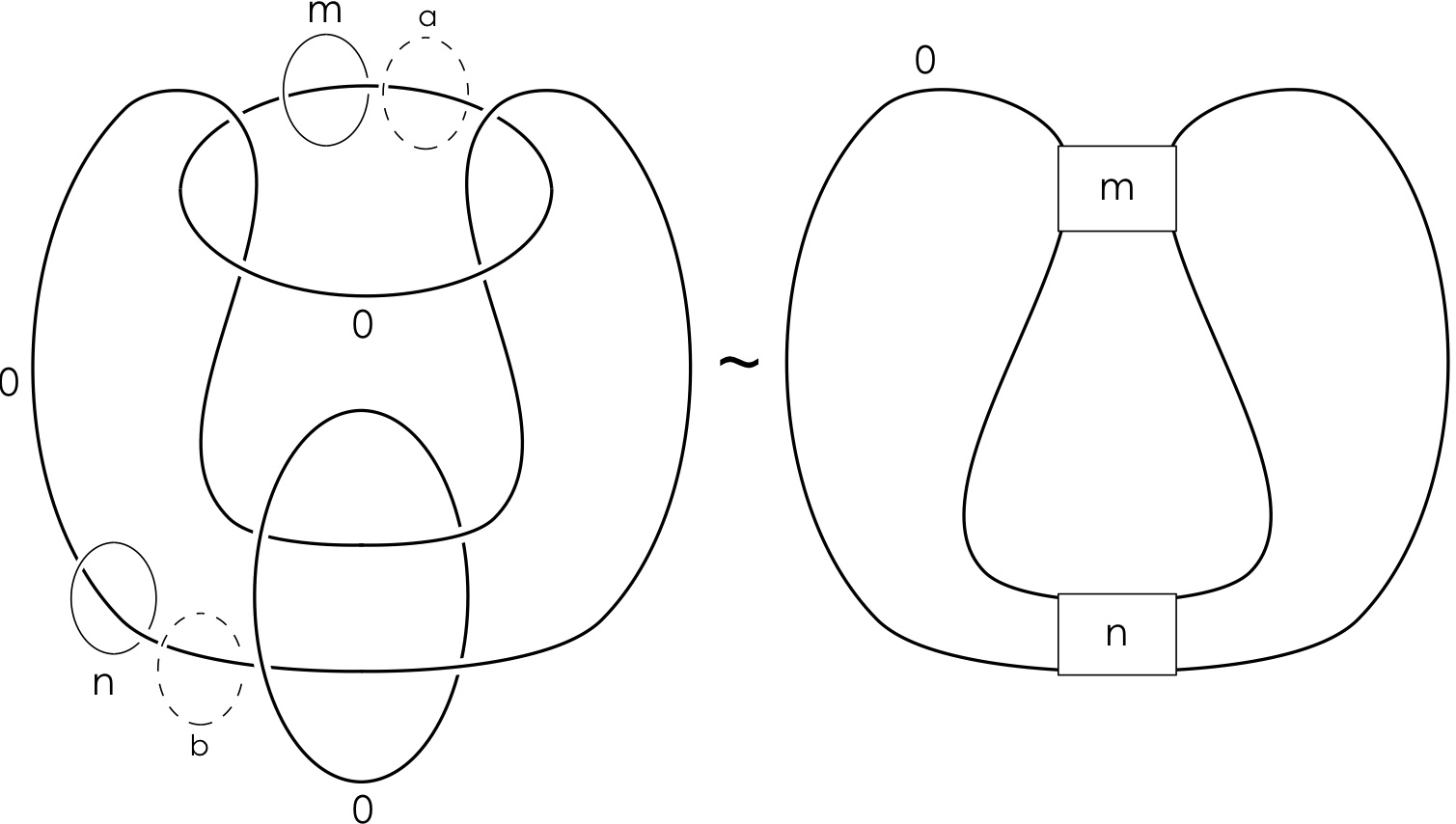}
			\caption{(a) Two surgeries in $T^3$ produce $Y_{m,n}$ \qquad (b) $Y_{m,n}$}
			\label{fig:XmnUnuTT}
			\end{center}
			\end{figure}
			
\noindent Since all of the link coefficients are integral, we can consider $Y_{m,n}$ as the boundary of some 4-manifold, say $U_{m,n}$.  After sliding one of $U_{m,n}$'s 0-framed 2-handles over and then off of the m- and n-framed components and then removing hopf pairs, we obtain figure \ref{fig:XmnUnuTT}(b).  Figure \ref{fig:XmnUnuTT}(b) (viewed again as a 3-manifold surgery diagram) with $m=0$ or $n=0$ is of course $S^2 \times S^1$.  Hence, for any such $(m,n)$ pair, 
$$Y_{m,n} \times S^1 \setminus \NT = S^2 \times S^1 \times S^1 \setminus \NT = A$$
 (see for instance:  \cite{FS10} Lemma 1).  On the other hand, $Y_{m,n} \ncong S^2 \times S^1$ for any choice of a nonzero pair $(m,n)$ since in that case $Y_{m,n}$ is not the unknot.
\end{proof}

Second, by gluing $\NT$ onto any of the surgery manifolds of theorem \ref{thm:totosurgery}, we can recast the union as a $T^3$-bundle over $S^1$, that is, a mapping torus of the form:
$$M_{\phi} \seteq \frac{I \times T^3}{(0,x) \sim (1, \phi (x))}$$
for some diffeomorphism $\phi : T^3 \longrightarrow T^3$.  For instance, $\NT  \cup  \toto = T^4 = T^3 \times S^1 = M_I$, for $I$ the identity map.  Furthermore, $A \cup \NT \approx S^2 \times T^2$, so $A$ does not correspond to a $T^3$-bundle over $S^1$. 

Now mapping tori such as these are precisely the kind that arise in the classic homotopy 4-sphere construction of Cappell-Shaneson  \cite{CS76b, CS76a}.  However, such monodromies $\phi$ obtained here from $\toto$ are \emph{not} restricted to $\SLIII$ and do not satisfy an additional condition of Cappell-Shaneson ($Det(\phi-I) = \pm1$, \cite{CS76b}), so surgery along the $0$-section in any of our mapping tori will not produce one of their homotopy spheres directly.

Now in his Ph.D. thesis (\cite{Na10}), the author of the present paper proved the following:
\begin{thm}\label{thm:MapTori}  Any Cappell-Shaneson mapping torus $M_{\phi}$ can be obtained by some sequence of surgeries along 2-tori in the fiber of the trivial bundle $T^4 = T^3 \times S^1$.  \hfill $\Box$
\end{thm}

We contend however that theorem \ref{thm:MapTori} is still not enough to immediately trivialize even one of the examples $\cS_{(m,n,m',n')}$ by relating these spheres to any of those within the Cappell-Shaneson collection that are now known to be standard (most recently due to Akbulut \cite{Ak09}  and independently Gompf \cite{Go09}).

The correspondences and differences can be seen as follows:
Performing $(\frac{1}{q})$-surgeries along product 2-tori embedded in the $T^3$-fiber of any mapping torus $M_A$ alters its monodromy by left multiplication with the surgery matrix (as in \cite{Go09}, and one dimension lower in \cite{GS99} example 8.2.4).   Unlike proposition \ref{prop:surgRecast} above, this time we factor the trivial fibration $T^4 = T^3 \times S^1$ as $(a \times b \times c) \times S^1 _d$.  Recall that the surgeries on $\toto$ producing the manifolds whose union is a homotopy $S^4$ , say for instance $X_{1,1}$, have surgery curves $\mu + a$ or $\mu + b$, respectively (in the basis $\left\{a, b, \mu \right\}$, $\mu$ the meridian of the torus).  Hence, back within the mapping-torus framework a $+1$-surgery along each of these two tori in these directions would give monodromy-multiplying matrices
$$R_{12} \seteq \begin{pmatrix}
1  &  1 & 0 \\
0  &  1  & 0 \\
0  &  0  & 1
\end{pmatrix}, \ R_{21} \seteq 
\begin{pmatrix}
1  &  0  & 0 \\
1 &  1  & 0 \\
0  &  0  & 1
\end{pmatrix}$$
respectively (now in the basis $\left\{a, b, c \right\}$).

The $X_{1,1}$ surgery manifolds then translate to mapping tori $M_{R_{12}R_{21}I} =  M_A$ where
$$A = \begin{pmatrix}
1  &  1 & 0 \\
0  &  1  & 0 \\
0  &  0  & 1
\end{pmatrix}
\begin{pmatrix}
1  &  0  & 0 \\
1 &  1  & 0 \\
0  &  0  & 1
\end{pmatrix} =
\begin{pmatrix}
2 & 1 & 0 \\
1 & 1 & 0 \\
0 & 0 & 1
\end{pmatrix}
.$$
However, the full range of Cappell-Shaneson mapping tori $M_{\phi}$ also involve surgering in the direction of the third $T^3$ basis factor, so for instance, $b_1 = 1$ in the Cappell-Shaneson case vs. $b_1(X_{m,n})=2$ here.  Moreover, any $X_{m,n}$ with either $m$ or $n \neq \pm1$ no longer even gives a $T^3$-bundle over $S^1$ when $\NT$ is added back in:  
In the case of a $(\pm \frac{1}{1})$-surgery, the diffeomorphism of the surgery torus ``lines up'' with a diffeomorphism of a fiber torus, but in a general $(\frac{m}{1})$-surgery ($m\neq \pm1)$ this correspondence fails.  Thus in general, the complement of $\NT$ in a true Cappell-Shaneson $M_{\phi}$ is \emph{not} an $X_{m,n}$ surgery manifold.  

One single surgery of the $X_{m,n}$ type is enough to derail $\toto$ from the Cappell-Shaneson track.  Note that the surgery is still reversible.  The point is that it is \textit{not} reversible or achievable by torus surgeries obtained from product-framed ($\frac{1}{q})$-surgeries on product tori in the $T^3$ fiber -- the sort used in theorem \ref{thm:MapTori}.

\subsection{Conclusion}
The combined above results should indicate that once again $\toto$ itself remains an important component to a diverse range of 4-manifold constructions, surgeries along tori playing a role in each case.  In fact, a slight alteration of the gluing $\varphi$ in the $\toto$ unions above into a fixed-point-free involution of $\partial (\toto) = \totopar$ allows for the construction of a fixed-point-free involution on the resulting homotopy sphere.  From this, homotopy $\nRPP$'s can then be constructed (given that two identical pairs of surgeries were performed) --- again with $\toto$ playing the role of the fundamental piece to the construction.

Finally, despite the correlations between the two realms, it does not appear that any of the homotopy spheres $\cS_{(m,n,m',n')}$ (parameters in $\nZ^{\neq 0}$) actually relate directly to any member of the Cappell-Shaneson collection (do they even contain fibered 2-spheres?), nor does it seem that Gompf's trick of fishtail-surgery \cite{Go09} would help in trivializing them.  The question remains then:
\begin{question}  Are the $\cS_{(m,n,m',n')}$ homotopy spheres standard?
\end{question}

\vspace{1cm}

\addcontentsline{toc}{chapter}{BIBLIOGRAPHY \vspace{0.5\baselineskip}} \thispagestyle{empty}

\section*{Acknowledgements}
The bulk of this paper originates from the author's Ph.D. thesis \cite{Na10} completed at the University of California Irvine under the supervision of Professor Ron Stern, whose insight, help, and guidance is very much appreciated.  The author would also like to thank Andr\'{a}s Stipsicz for very helpful feedback and for kindly pointing out an error in an earlier draft.

\bibliography{nashbibsub.bib}

\end{document}